\documentclass[11pt, reqno]{art-hznu}

\usepackage{amsmath, amsthm, amscd, amsfonts, amssymb, graphicx, color}
\usepackage{enumerate}
\usepackage{enumitem}
\RequirePackage[backref,colorlinks=true]{hyperref}
\newtheorem{theorem}{Theorem}[section]
\newtheorem{lemma}[theorem]{Lemma}

\newtheorem{corollary}[theorem]{Corollary}

\theoremstyle{definition}

\theoremstyle{remark}

\numberwithin{equation}{section}

\begin{document}

\title{Optimal bounds for a Gaussian Arithmetic-Geometric type mean by quadratic and contraharmonic means}

\author[]{Junxuan Shen}
\ead{2014hfser@gmail.com}

\address{Hangzhou Foreign Languages School, Hangzhou 310023, China}
\setcounter{page}{1}

\begin{abstract}
In this paper, we present the best possible parameters $\alpha_i, \beta_i\ (i=1,2,3)$ and $\alpha_4,\beta_4\in(1/2,1)$ such that the double inequalities 
\begin{align*}
\alpha_1Q(a,b)+(1-\alpha_1)C(a,b)&<AG_{Q,C}(a,b)<\beta_1Q(a,b)+(1-\beta_1)C(a,b),\\
\qquad\ Q^{\alpha_2}(a,b)C^{1-\alpha_2}(a,b)&<AG_{Q,C}(a,b)<Q^{\beta_2}(a,b)C^{1-\beta_2}(a,b),\\
\frac{Q(a,b)C(a,b)}{\alpha_3Q(a,b)+(1-\alpha_3)C(a,b)}&<AG_{Q,C}(a,b)<\frac{Q(a,b)C(a,b)}{\beta_3Q(a,b)+(1-\beta_3)C(a,b)},\\
C\left(\sqrt{\alpha_4a^2+(1-\alpha_4)b^2},\sqrt{(1-\alpha_4)a^2+\alpha_4b^2}\right)&<AG_{Q,C}(a,b)<C\left(\sqrt{\beta_4a^2+(1-\beta_4)b^2},\sqrt{(1-\beta_4)a^2+\beta_4b^2}\right)
\end{align*}
hold for all $a, b>0$ with $a\neq b$, where $Q(a,b)$, $C(a,b)$ and $AG(a,b)$ are the quadratic, contraharmonic and Arithmetic-Geometric means, and $AG_{Q,C}(a,b)=AG[Q(a,b),C(a,b)]$. As consequences, we present new bounds for the complete elliptic integral of the first kind.
\end{abstract}

\maketitle

\section{Introduction}
The classical arithmetic-geometric mean $AG(a,b)$ of two positive numbers $a$ and $b$ is defined by starting with $a_0=a,b_0=b$ and then iterating 
\begin{equation}\label{eq:1.1}
a_{n+1}=A(a_n,b_n),\qquad b_{n+1}=G(a_n,b_n)
\end{equation}for $n\in\mathbb{N}$ until two sequences $\{a_n\}$ and $\{b_n\}$ converge to the same number, where $A(a,b)=(a+b)/2$ and $G(a,b)=\sqrt{ab}$ are arithmetic and geometric means, respectively.

The well-known Gauss identity \cite{AVV} shows that
\begin{equation}\label{eq:1.2}
AG(1,r)\mathcal{K}(\sqrt{1-r^2})=\frac{\pi}{2}
\end{equation}
for $r\in(0,1)$, where $\mathcal{K}(r)=\int^{\pi/2}_{0}(1-r^2\sin^2t)^{-1/2}dt$ ($r\in[0,1)$) is the complete elliptic integral of the first kind. By use of the homogeneity of \eqref{eq:1.1} and \eqref{eq:1.2}, $AG(a,b)$ can be written explicitly as 
\begin{equation}\label{eq:1.3}
AG(a,b)=\frac{\pi}{2\int^{\pi/2}_{0}dt/\sqrt{a^2\cos^2t+b^2\sin^2t}}.
\end{equation}

Let $r\in(0,1)$ and $r'=\sqrt{1-r^2}$, then the complete elliptic integral of the second kind is given by $\mathcal{E}(r)=\int^{\pi/2}_{0}(1-r^2\sin^2t)^{1/2}dt$. We clearly see that $\mathcal{K}(r)$ is strictly increasing from $(0,1)$ onto $(\pi/2, +\infty)$ and $\mathcal{E}(r)$ is strictly decreasing from $(0, 1)$ onto $(1, \pi/2)$. Moreover, $\mathcal{K}(r)$ and $\mathcal{E}(r)$ satisfy the following Landen
identities and derivatives formulas (see \cite[Appendix E, p.474-475]{AVV})

\begin{align*}
\mathcal{K}\left(\frac{2\sqrt{r}}{1+r}\right)=(1+r)\mathcal{K}(r),\quad &\mathcal{E}\left(\frac{2\sqrt{r}}{1+r}\right)=\frac{2\mathcal{E}-r'^2\mathcal{K}}{1+r},\\
\frac{d\mathcal{K}(r)}{dr}=\frac{\mathcal{E}(r)-r'^2\mathcal{K}(r)}{rr'^2},\quad & \frac{d\mathcal{E}(r)}{dr}=\frac{\mathcal{E}(r)-\mathcal{K}(r)}{r},\\
\frac{d\left[\mathcal{E}(r)-r'^2\mathcal{K}(r)\right]}{dr}=r\mathcal{K}(r),\quad &\frac{d\left[\mathcal{K}(r)-\mathcal{E}(r)\right]}{dr}=\frac{r\mathcal{E}(r)}{r'^2},
\end{align*}

Two special values $\mathcal{K}(\sqrt{2}/2)$ and $\mathcal{E}{\sqrt{2}/2}$ will be used later, which can be expressed as (see \cite[Theorem 1.7]{AVV})
\begin{equation*}
\mathcal{K}(\sqrt{2}/2)=\frac{\Gamma^2(1/4)}{4\sqrt{\pi}}=1.854\cdots,\quad \mathcal{E}(\sqrt{2}/2)=\frac{4\Gamma^2(3/4)+\Gamma^2(1/4)}{8\sqrt{\pi}}=1.350\cdots,
\end{equation*}
where $\Gamma(x)=\int^\infty_0t^{x-1}e^{-t}dt$ is the classical Euler's gamma function.

The special bivariate mean $AG_{X,Y}(a,b)$ derived from Arithmetric-Geometric mean for any bivariate means $X(a, b)$ and
$Y(a, b)$ of positive numbers a, b is given by 
\begin{equation}\label{eq:1.4}
AG_{X,Y}(a,b)=AG[X(a,b),Y(a,b)],
\end{equation}
which is called a Arithmetric-Geometric type mean. We denote the pairs of means $\{X,Y\}$ the generating means of the Arithmetric-Geometric type mean defined in \eqref{eq:1.4}.

It is well known that the elliptic elliptic integrals $\mathcal{K}(r)$ and $\mathcal{E}(r)$ and the Gaussian arithmetic-geometric mean $AG(a,b)$ have many applications in mathematics, physics, mechanics, and engineering \cite{Ca,EBJ,Le,Ho,MF,Ma}. Recently, the Arithmetric-Geometric mean has been the subject of intensive research. The double inequalities
\begin{align}\label{eq:1.5}
L_{-1}(a,b)=L(a,b)&<AG(a,b)<L^{2/3}(a^{3/2},b^{3/2}),\\ \label{eq:1.6}
\frac{1+\sqrt{r}}{2}AG(1,\sqrt{r})&<AG(1,r)<\frac{\pi}{2\log(4/r)}
\end{align}hold for all $a,b>0$ with $a\neq b$, where $L_p(a,b)=\left[(b^{p+1}-a^{p+1})/((p+1)(b-a))\right]^{1/p}(p\neq-1,0)$, $L_{-1}(a,b)=L(a,b)=(b-a)/(\log b-\log a)$ and $L_0(a,b)=I(a,b)=(b^b/a^a)^{1/(b-a)}/e$. The left inequality of \eqref{eq:1.5} was first proposed by Carlson and Vuorinen \cite{CV} and also was proved by different methods in \cite{Sa,NS,Ya}. Vamanamurthy and Vuorinen \cite{VV} proved that $AG(a,b)<(\pi/2)L(a,b)$ for all $a,b>0$ with $a\neq b$. The second inequality of \eqref{eq:1.5} was proved by Borwein \cite{BB} and Yang \cite{Ya}. Very recently, Ding and Zhao \cite{DZ} showed that $L_{-1}(a,b)<AG(a,b)<L_{-1/2}(a,b)$ for all $a,b>0$ with $a\neq b$, where $L_{-1}(a,b)$ and $L_{-1/2}(a,b)$ are the best possible lower and upper generalized logarithmic mean bounds, respectively. 

In \cite{Ku}, K\"{u}hnau refined the double inequality \eqref{eq:1.6} and obtained the improved upper bound $\pi(1-r^2/9)/[2\log(4/r)]$. Qiu and Vamanamurthy \cite{QV} presented the new lower and upper bounds for $AG(1,r)$ with $r\in(0,1)$, which are $4\pi/[(9-r^2)\log(4/r)]$ and $(9-r^2)\pi/[18.192\log(4/r)]$, respectively.

Alzer and Qiu \cite{AQ} proved that the double inequality
\begin{equation}\label{eq:1.7}
\frac{1}{\lambda/L(a,b)+(1-\lambda)/A(a,b)}<AG(a,b)<\frac{1}{\mu/L(a,b)+(1-\mu)/A(a,b)}
\end{equation}holds for all $a,b>0$ with $a\neq b$ if and only if $\lambda\geq3/4$ and $\mu\leq2/\pi$.

Chu and Wang \cite{CW} proved that the double inequality 
\begin{equation}\label{eq:1.8}
S_{p}(a,b)<AG(a,b)<S_{q}(a,b),
\end{equation}
for all $a,b>0$ with $a\neq b$ if and only if $p\leq1/2,q\geq1$, where $S_p(a,b)=[(a^{p-1}+b^{p-1})/(a+b)]^{1/(p-2)}$ $(p\neq2)$
and $S_2(a,b)=(a^ab^b)^{1/(a+b)}$ is the $p$th Gini mean of $a$ and $b$. In \cite{YSC}, Yang et al. proved that inequalities 
\begin{align}\label{eq:1.9}
\begin{split}
S_{7/4,-1/4}(a,b)&<AG(a,b)<A^{1/4}(a,b)L^{3/4}(a,b),\\
AG(a,b)&<\sqrt{S_{p,1}(a,b)S_{1-p,1}(a,b)}
\end{split}
\end{align}
hold for all $p\in(1/2,1)$ and $a,b>0$ with $a\neq b$, where $S_{p,q}(a,b)=\left[q(a^p-b^p)/(p(a^q-b^q))\right]^{1/(p-q)}$ is the Stolarsky
mean \cite{YCZ} of $a$ and $b$.

Very recently, optimal bounds for $AG_{A,Q}(a,b)$ by several convex combinations of their generating means were established.
Explicitly, Wang et al. \cite{WQC} presented  the best possible parameters $\alpha_i, \beta_i\ (i=1,2,3)$ such that the double inequalities
\begin{align*}
Q^{\alpha_1}(a,b)A^{1-\alpha_1}(a,b)&<AG_{A,Q}(a,b)<Q^{\beta_1}(a,b)A^{1-\beta_1}(a,b),\\
\alpha_2Q(a,b)+(1-\alpha_2)A(a,b)&<AG_{A,Q}(a,b)<\beta_2Q(a,b)+(1-\beta_2)A(a,b),\\
Q\left[\alpha_3a+(1-\alpha_3)b,(1-\alpha_3)a+\alpha_3b\right]&<AG_{A,Q}(a,b)<Q\left[\beta_3a+(1-\beta_3)b,(1-\beta_3)a+\beta_3b\right]
\end{align*}
hold for all $a,b>0$ with $a\neq b$, where $Q(a,b)=\sqrt{(a^2+b^2)/2}$ is the quadratic mean of $a$ and $b$.

Let $C(a,b)=(a^2+b^2)/(a+b)$ is the contraharmonic mean, then it is easy to verify that the function $\zeta(x)=C\left[\sqrt{xa^2+(1-x)b^2},\sqrt{(1-x)a^2+xb^2}\right]$ is continuous and strictly increasing on $[1/2,1]$. Note that
\begin{equation}\label{eq:1.10}
\zeta(1/2)=Q(a,b)<\min\{Q(a,b),C(a,b)\}<AG_{Q,C}(a,b)<\max\{Q(a,b),C(a,b)\}=C(a,b)=\zeta(1).
\end{equation}

Motivated by inequality \eqref{eq:1.10} and the results of \cite{WQC}, it is natural to ask what are the best possible parameters $\alpha_i, \beta_i\ (i=1,2,3)$ and $\alpha_4,\beta_4\in(1/2,1)$ such that the double inequalities
\begin{align*}
\alpha_1Q(a,b)+(1-\alpha_1)C(a,b)&<AG_{Q,C}(a,b)<\beta_1Q(a,b)+(1-\beta_1)C(a,b),\\
Q^{\alpha_2}(a,b)C^{1-\alpha_2}(a,b)&<AG_{Q,C}(a,b)<Q^{\beta_2}(a,b)C^{1-\beta_2}(a,b),\\
\frac{Q(a,b)C(a,b)}{\alpha_3Q(a,b)+(1-\alpha_3)C(a,b)}&<AG_{Q,C}(a,b)<\frac{Q(a,b)C(a,b)}{\beta_3Q(a,b)+(1-\beta_3)C(a,b)},\\
C\left(\sqrt{\alpha_4a^2+(1-\alpha_4)b^2},\sqrt{(1-\alpha_4)a^2+\alpha_4b^2}\right)&<AG_{Q,C}(a,b)\\
&\hspace{1.5cm}<C\left(\sqrt{\beta_4a^2+(1-\beta_4)b^2},\sqrt{(1-\beta_4)a^2+\beta_4b^2}\right)
\end{align*}
hold for all $a,b>0$ with $a\neq b$ The main purpose of this paper is to answer this question.

\section{Lemmas}

In order to prove the desired theorem, we need several lemmas which we present this section.

\begin{lemma}\label{lm2.1}(see \cite[Theorem 1.25]{AVV}) For $-\infty<a<b<\infty$, let $f,g:[a,b]\rightarrow\mathbb{R}$ be continuous on $[a,b]$, and be differetiable on $(a,b)$, let $g'(x)\neq0$ on (a,b). If $f'(x)/g'(x)$ is increasing (decreasing) on $(a,b)$, then so are
\begin{equation*}
\frac{f(x)-f(a)}{g(x)-g(a)}\quad\text{and}\quad \frac{f(x)-f(b)}{g(x)-g(b)}.
\end{equation*}
If $f'(x)/g'(x)$ is strictly monotone, then the monotonicity in the conclusion is also strict.
\end{lemma}

\begin{lemma}\label{lm2.2}
\begin{enumerate}[itemindent=-0.5em,label=(\arabic*)]
\item $[\mathcal{E}(r)-r'^2\mathcal{K}(r)]/r^2$ is strictly increasing from $(0,1)$ onto $(\pi/4,1)$; \label{it2.2.1}
\item For each $c\in [1/2,\infty)$, the function $r'^c\mathcal{K}(r)$ is decreasing from $[0,1)$ onto $(0,\frac{\pi}{2}]$; \label{it2.2.2}
\item $[\mathcal{E}(r)-r'^2\mathcal{K}(r)]/[r^2\mathcal{K}(r)]$ is strictly decreasing on $(0,1)$;\label{it2.2.3}
\item$[\mathcal{K}(r)-\mathcal{E}(r)]/r^2$ is strictly increasing on $(0,1)$.\label{it2.2.4}
\end{enumerate}
\end{lemma}
\begin{proof}
Parts (1)-(4) follow from \cite[Exercise 3.43 (11) and (46), Theorem 3.21 (1) and (7)]{AVV}.
\end{proof}

\begin{lemma}\label{lm2.3}
Let $\delta_1=(2+\sqrt{2})\left(1-\pi/[2\mathcal{K}(\sqrt{2}/2)]\right)=0.5216\cdots$ and 
\begin{equation*}
f(r)=\frac{1-\pi/[2\mathcal{K}(r)]}{1-\sqrt{1-r^2}},
\end{equation*}
then $f(r)$ is strictly increasing from $(0,\sqrt{2}/2)$ onto $(1/2,\delta_1)$.
\end{lemma}

\begin{proof}
Let $f_1(r)=1-\pi/[2\mathcal{K}(r)]$ and $f_2(r)=1-\sqrt{1-r^2}$, then we clearly see that
$f_1(0)=f_2(0)=0, f(r)=f_1(r)/f_2(r)$ and
\begin{equation}\label{eq:2.1}
\frac{f'_1(r)}{f'_2(r)}=\frac{\pi}{2}\cdot\frac{\mathcal{E}(r)-r'^2\mathcal{K}(r)}{r^2}\cdot\frac{1}{\left[r'^{1/2}\mathcal{K}(r)\right]^2}.\end{equation}

It follows from Lemma \ref{lm2.2} \ref{it2.2.1} and \ref{it2.2.2} that $[\mathcal{E}(r)-r'^2\mathcal{K}(r)]/r^2$ and $1/\left[r'^{1/2}\mathcal{K}(r)\right]^2$ are strictly increasing on $(0,1)$. This conjunction with \eqref{eq:2.1} implies that $f'_1(r)/f'_2(r)$ is strictly increasing on $(0,1)$. 

Therefore, Lemma \ref{lm2.3} follows immediately from Lemma \ref{lm2.1} and the limiting values $f(0^+)=1/2$ and $f(\sqrt{2}/2^-)=\delta_1$.
\end{proof}

\begin{lemma}\label{lm2.4}
Let $\delta_2=2\log[2\mathcal{K}(\sqrt{2}/2)/\pi]/\log2=0.4784\cdots$ and 
\begin{equation*}
g(r)=\frac{\log\pi/2-\log\mathcal{K}(r)}{\log\sqrt{1-r^2}},
\end{equation*}
then $g(r)$ is strictly decreasing from $(0,\sqrt{2}/2)$ onto $(\delta_2,1/2)$.
\end{lemma}

\begin{proof}
Let $g_1(r)=\log\pi/2-\log\mathcal{K}(r)$ and $g_2(r)=\log\sqrt{1-r^2}$, then we clearly see that
$g_1(0)=g_2(0)=0, g(r)=g_1(r)/g_2(r)$ and
\begin{equation}\label{eq:2.2}
\frac{g'_1(r)}{g'_2(r)}=\frac{\mathcal{E}(r)-r'^2\mathcal{K}(r)}{r^2\mathcal{K}(r)}.\end{equation}

It follows from Lemma \ref{lm2.2} \ref{it2.2.3} and \eqref{eq:2.2} that $g'_1(r)/g'_2(r)$ is strictly increasing on $(0,1)$. 

Therefore, Lemma \ref{lm2.4} follows from Lemma \ref{lm2.1} and the limiting values $g(0^+)=1/2$ and $g(\sqrt{2}/2^-)=\delta_2$.
\end{proof}

\begin{lemma}\label{lm2.5}
Let $\delta_3=(\sqrt{2}+1)\left(\sqrt{2}-[2\mathcal{K}(\sqrt{2}/2)]/\pi\right)=0.5646\cdots$ and 
\begin{equation*}
h(r)=\frac{2\sqrt{1-r^2}\mathcal{K}(r)/\pi-1}{\sqrt{1-r^2}-1},
\end{equation*}
then $h(r)$ is strictly increasing from $(0,\sqrt{2}/2)$ onto $(1/2,\delta_3)$.
\end{lemma}

\begin{proof}
Let $h_1(r)=2\sqrt{1-r^2}\mathcal{K}(r)/\pi-1$ and $h_2(r)=\sqrt{1-r^2}-1$, then it is easy to see that
$h_1(0)=h_2(0)=0, h(r)=h_1(r)/h_2(r)$ and
\begin{equation}\label{eq:2.3}
\frac{h'_1(r)}{h'_2(r)}=\frac{\mathcal{K}(r)-\mathcal{E}(r)}{r^2}.
\end{equation}

Lemma \ref{lm2.2} \ref{it2.2.4} and \eqref{eq:2.3} lead to the conclusion that $h'_1(r)/h'_2(r)$ is strictly increasing on $(0,\sqrt{2}/2)$.

Therefore, Lemma \ref{lm2.5} follows easily from Lemma \ref{lm2.1} and the limiting values $h(0^+)=1/2$ and $h(\sqrt{2}/2^-)=\delta_3$.
\end{proof}

\begin{lemma}\label{lm2.6}
The inequality 
\begin{equation*}
\mathcal{K}(r)>\frac{\pi}{2}\left(1+\frac{r^2}{4}\right)
\end{equation*}
holds for $r\in(0,1)$.
\end{lemma}

\begin{proof}
In order to prove this lemma, it suffices to show the inequality
\begin{equation}\label{eq:2.4}
\frac{2\mathcal{K}(r)/\pi-1}{r^2}>\frac{1}{4}
\end{equation}
holds for $r\in(0,1)$.

Let $\mu(r)=[2\mathcal{K}(r)/\pi-1]/r^2$, $\mu_1(r)=2\mathcal{K}(r)/\pi-1$ and $\mu_2(r)=r^2$, then we clearly see that
\begin{equation}\label{eq:2.5}
\mu_1(0)=\mu_2(0)=0,\quad \mu(r)=\frac{\mu_1(r)}{\mu_2(r)}.
\end{equation}

Taking the derivative of $\mu_1(r)$ and $\mu_2(r)$ yieds
\begin{equation}\label{eq:2.6}
\frac{\mu'_1(r)}{\mu'_2(r)}=\frac{1}{\pi}\cdot\frac{\mathcal{E}(r)-r'^2\mathcal{K}(r)}{r^2}\cdot\frac{1}{r'^2}.
\end{equation}

It follows from \eqref{eq:2.6} and Lemma \ref{lm2.2} \ref{it2.2.1} together with the monotonicity of $r'=\sqrt{1-r^2}$ that $\mu'_1(r)/\mu'_2(r)$ is strictly increasing on $(0,1)$. This conjunction with \eqref{eq:2.5} and Lemma \ref{lm2.1}  implies that $\mu(r)$ is strictly increasing on $(0,1)$.

Therefore, the desired inequality \eqref{eq:2.4} follows from $\mu(0^+)=1/4$ and the monotonicity of $\mu(r)$.
\end{proof}

\begin{lemma}\label{lm2.7}
 The function $\eta(r)=\left[\pi/\left(2\mathcal{K}(r)\right)+r^2/4-1\right]/r^4$ is strictly decreasing from $(0,\sqrt{2}/2)$ onto $(-7/2+2\pi/\mathcal{K}(\sqrt{2}/2),-5/64)$.
\end{lemma}

\begin{proof}
Let $\eta_1(r)=\pi/[2\mathcal{K}(r)]+r^2/4-1$ and $\eta_2(r)=r^4$, then it is easy to see that
$\eta_1(0)=\eta_2(0)=0, \eta(r)=\eta_1(r)/\eta_2(r)$ and
\begin{equation}\label{eq:2.7}
\frac{\eta'_1(r)}{\eta'_2(r)}=\frac{\eta_{11}(r)}{\eta_{22}(r)},
\end{equation}
where $\eta_{11}(r)=r^2r'^2\mathcal{K}^2(r)-\pi[\mathcal{E}(r)-r'^2\mathcal{K}(r)]$ and $\eta_{22}(r)=8r^4r'^2\mathcal{K}^2(r)$.

Observe that $\eta_{11}(0)=\eta_{22}(0)=0$. Taking the derivative of $\eta_{11}(r)$ and $\eta_{22}(r)$ yields
\begin{equation}\label{eq:2.8}
\frac{\eta'_{11}(r)}{\eta'_{22}(r)}=-\frac{\zeta_1(r)}{\zeta_2(r)},
\end{equation}
where
\begin{equation}\label{eq:2.9}
\zeta_1(r)=\frac{\pi-2\mathcal{E}(r)+2r^2\mathcal{K}(r)}{16r^2},\quad \zeta_2(r)=\mathcal{E}(r)+(1-2r^2)\mathcal{K}(r).
\end{equation}

An easy computation leads to  
\begin{equation}\label{eq:2.10}
\zeta'_1(r)=\frac{\mathcal{E}(r)-\pi r'^2+r'^4\mathcal{K}(r)}{8r^3(1-r^2)}
\end{equation}
for $r\in(0,\sqrt{2}/2)$.

It follows from \cite[Corollary 2.7]{JDZ} that $2\mathcal{E}(r)/\pi>1-r^2/4-r^4/8$ for $r\in(0,\sqrt{2}/2)$. This can be rewritten as
\begin{equation}\label{eq:2.11}
\mathcal{E}(r)-\pi r'^2>\frac{\pi}{2}\left(-1+\frac{7r^2}{4}-\frac{r^4}{8}\right)
\end{equation}for $r\in(0,\sqrt{2}/2)$.

Lemma \ref{lm2.6} and \eqref{eq:2.11} lead to the conclusion that 
\begin{equation*}
\mathcal{E}(r)-\pi r'^2+r'^4\mathcal{K}(r)>\frac{\pi}{2}\left(-1+\frac{7r^2}{4}-\frac{r^4}{8}\right)+\frac{\pi}{2}(1-r^2)^2\left(1+\frac{r^2}{4}\right)=\frac{\pi r^4(3+2r^2)}{16}>0
\end{equation*}for $r\in(0,\sqrt{2}/2)$. This conjunction with \eqref{eq:2.10} implies that $\zeta_1(r)$ is strictly increasing on $(0,\sqrt{2}/2)$. Since $\zeta_2(r)$ can be rewritten as $\mathcal{E}(r)-\mathcal{K}(r)+2r'^2\mathcal{K}(r)$, we conclude easily from Lemma \ref{lm2.2} \ref{it2.2.2} and \ref{it2.2.4} that $\zeta_2(r)$ is strictly decreasing on $(0,\sqrt{2}/2)$. 

Moreover, it follows easily from \eqref{eq:2.9} that $\zeta_1(r)>0$ and $\zeta_2(r)>0$ for $r\in(0,\sqrt{2}/2)$. This conjunction with 
\eqref{eq:2.8} together with the monotonicity of $\zeta_1(r)$ and $\zeta_2(r)$ implies that $\eta'_{11}/(r)\eta'_{22}(r)$ is strictly decreasing on $(0,\sqrt{2}/2)$.

Therefore, Lemma \ref{lm2.4} follows immediately from \eqref{eq:2.7} and Lemma \ref{lm2.1} together with the limiting values 
$\eta(0^+)=-5/64$ and $\eta(\sqrt{2}/2^-)=-7/2+2\pi/\mathcal{K}(\sqrt{2}/2)$.

\end{proof}

Note that $-7/2+2\pi/\mathcal{K}(\sqrt{2}/2)=-0.111148<-7/64$, then the following corollary follows directly from Lemma \ref{lm2.7}.

\begin{corollary}\label{cl2.8}
The double inequality 
\begin{equation*}
1-\frac{r^2}{4}-\frac{7r^4}{64}<\frac{\pi}{2\mathcal{K}(r)}<1-\frac{r^2}{4}-\frac{5r^4}{64}
\end{equation*}holds for $r\in(0,\sqrt{2}/2)$.
\end{corollary}

\begin{lemma}\label{lm2.9}
Let $\delta_4=4\sqrt{2\mathcal{K}^2(\sqrt{2}/2)\left[\pi^2-\mathcal{K}^2(\sqrt{2}/2)\right]}/\pi=1.8389\cdots$, $\lambda\in(0,2]$ and  
\begin{equation*}
\Phi_\lambda(r)=\frac{\sqrt{1+\lambda rr'}-\sqrt{1-\lambda rr'}}{\lambda r},
\end{equation*}
which is defined as in \cite[Lemma 2.8]{JDZ}, then the following statements are true:
\begin{enumerate}[itemindent=-0.5em,label=(\arabic*)]
\item $\Phi_{\delta_4}(r)>1-r^2/4$ holds for $r\in(0,33/50)$;\label{it2.8.1}
\item $\mathcal{K}(r)/\sqrt{r}$ is strictly decreasing on $(0,\sqrt{2}/2)$;\label{it2.8.2}
\item $\sqrt{r}\Phi_{\delta_4}(r)$ is strictly decreasing on $(33/50,\sqrt{2}/2)$.\label{it2.8.3}
\end{enumerate}
\end{lemma}

\begin{proof}
(1) In order to prove that $\Phi_{\delta_4}(r)>1-r^2/4$ for $r\in(0,33/50)$, by squaring both sides of the inequality and simplifying, it suffices to show
\begin{equation}\label{eq:2.12}
\sqrt{1-\delta_4^2r^2(1-r^2)}<1-\frac{\delta_4^2}{2}r^2\left(1-\frac{r^2}{4}\right)^2
\end{equation}holds for $r\in(0,33/50)$.

The difference of both sides squares of \eqref{eq:2.12} leads to
\begin{equation} \label{eq:2.13} 
1-\delta_4^2r^2(1-r^2)-\left[1-\frac{\delta_4^2}{2}r^2\left(1-\frac{r^2}{4}\right)^2\right]^2=-\frac{\delta_4^2r^4}{1024}\nu(r),
\end{equation}
where 
\begin{equation*}
\nu(r)=\delta_4^2r^8-16\delta_4^2r^6+96\delta_4^2r^4-256\delta_4^2r^2-64r^2+256\delta_4^2-512.
\end{equation*}

An easy calculation yields 
\begin{align} \label{eq:2.14} 
\nu(33/50)&=5.9588\cdots,\\ \label{eq:2.15} 
\nu'(r)&=-8r[\delta_4^2(4-r^2)^3+16]<0. 
\end{align}
It follows from \eqref{eq:2.14} and \eqref{eq:2.15} that $\nu(r)>0$ for $r\in(0,33/50)$.  This conjunction with \eqref{eq:2.13} completes the proof of Lemma \ref{lm2.9} \ref{it2.8.1}.

\medskip

(2) It suffices to determine the sign of the derivate of $\mathcal{K}(r)/\sqrt{r}$. An easy computation yields
\begin{equation}\label{eq:2.16} 
\frac{d[\mathcal{K}(r)/\sqrt{r}]}{dr}=\frac{\xi(r)}{2r^{3/2}r'^2},
\end{equation}
where 
\begin{equation*}
\xi(r)=3[\mathcal{E}(r)-r'^2\mathcal{K}(r)]-\mathcal{E}(r).
\end{equation*}
It follows from Lemma \ref{lm2.2} \ref{it2.2.1} and the monotonicity of $\mathcal{E}(r)$ that $\xi(r)$ is strictly increasing on $(0,\sqrt{2}/2)$.
As a consequence, we obtain
\begin{equation}\label{eq:2.17}
\xi(r)<\xi(\sqrt{2}/2)=-0.07982\cdots<0
\end{equation}for $r\in(0,\sqrt{2}/2)$.

Therefore, we conclude from \eqref{eq:2.16} and \eqref{eq:2.17} that $\mathcal{K}(r)/\sqrt{r}$ is strictly decreasing on $(0,\sqrt{2}/2)$.

\medskip

(3) Let $\omega_1(r)=\sqrt{1+\delta_4rr'}$ and $\omega_2(r)=\sqrt{1-\delta_4rr'}$, then we clearly see that $\omega_1(r)>\omega_2(r)>0$ for $r\in(0,\sqrt{2}/2)$ and $\Phi_{\delta_4}(r)=[\omega_1(r)-\omega_2(r)]/(\delta_4 r)$.

Easy computations lead to
\begin{align}\label{eq:2.18} 
\frac{d[\omega_1(r)-\omega_2(r)]}{dr}&=\frac{\delta_4(1-2r^2)}{2\sqrt{1-r^2}}\left[\frac{1}{\omega_1(r)}+\frac{1}{\omega_2(r)}\right]>0,\\ \label{eq:2.19} 
\frac{d[1/\omega_1(r)+1/\omega_2(r)]}{dr}&=\frac{\delta_4(1-2r^2)}{2\sqrt{1-r^2}}\left[\frac{1}{\omega_2^3(r)}-\frac{1}{\omega_1^3(r)}\right]>0
\end{align}
for $r\in(0,\sqrt{2}/2)$.  It follows from \eqref{eq:2.18} and \eqref{eq:2.19} that $\omega_1(r)-\omega_2(r)$ and $1/\omega_1(r)+1/\omega_2(r)$ are strictly increasing on $r\in(0,\sqrt{2}/2)$. Moreover,  the monotonicity property of composite function leads to the conculsion that $(1-2r^2)/\sqrt{1-r^2}=2\sqrt{1-r^2}-1/\sqrt{1-r^2}$ is strictly decreasing on $(0,\sqrt{2}/2)$.  These properties imply that 
\begin{align}\nonumber
\frac{d[\sqrt{r}\Phi_{\delta_4}(r)]}{dr}&=\frac{1}{2\sqrt{r}}\left[-\frac{\omega_1(r)-\omega_2(r)}{\delta_4 r}+\frac{1-2r^2}{\sqrt{1-r^2}}\left(\frac{1}{\omega_1(r)}+\frac{1}{\omega_2(r)}\right)\right]\\ \nonumber
&<\frac{1}{2\sqrt{r}}\left[-\frac{\sqrt{2}[\omega_1(33/50)-\omega_2(33/50)]}{\delta_4} +\frac{1-2\times(33/50)^2}{\sqrt{1-(33/50)^2}}\left(\frac{1}{\omega_1(\sqrt{2}/2)}+\frac{1}{\omega_2(\sqrt{2}/2)}\right)\right]\\ \nonumber
&<\frac{1}{2\sqrt{r}}\left(-\frac{83}{100}+\frac{73}{100}\right)\\ \label{eq:2.20}
&=-\frac{1}{20\sqrt{r}}<0
\end{align}for $r\in(33/50,\sqrt{2}/2)$.

Therefore, Lemma \ref{lm2.9} \ref{it2.8.3} follows directly from \eqref{eq:2.20}.
\end{proof}

\bigskip

\section{Main results}

\begin{theorem}\label{tm3.1}
The double inequality
\begin{equation*}
\alpha_1Q(a,b)+(1-\alpha_1)C(a,b)<AG_{Q,C}(a,b)<\beta_1Q(a,b)+(1-\beta_1)C(a,b)
\end{equation*}
holds for all $a,b>0$ with $a\neq b$ if and only if $\alpha_1\geq\delta_1=0.5216\cdots$ and $\beta_1\leq1/2$.
\end{theorem}

\begin{proof}
Since $Q(a,b),C(a,b)$ and $AG(a,b)$ are symmetric and homogeneous of degree 1, without loss of generality, we may assume that
$a>b>0$. Let $r=(a-b)/\sqrt{2(a^2+b^2)}\in(0,\sqrt{2}/2)$, then we clearly see from \eqref{eq:1.2} and \eqref{eq:1.3} together with the definition of $Q(a,b)$ and $C(a,b)$ that 
\begin{align}\label{eq:3.1}
Q(a,b)&=\sqrt{1-r^2}C(a,b),\\ \label{eq:3.2}
AG_{Q,C}(a,b)&=\frac{\pi C(a,b)}{2\mathcal{K}(r).}
\end{align}

It follows from \eqref{eq:3.1} and \eqref{eq:3.2} that
\begin{equation}\label{eq:3.3}
\frac{C(a,b)-AG_{Q,C}(a,b)}{C(a,b)-Q(a,b)}=\frac{1-\pi/[2\mathcal{K}(r)]}{1-\sqrt{1-r^2}}=f(r),
\end{equation}
where $f(r)$ is defined as in Lemma \ref{lm2.3}.

Therefore, Theorem \ref{tm3.1} follows easily from \eqref{eq:3.3} and Lemma \ref{lm2.3} .

\end{proof}

\begin{theorem}\label{tm3.2}
The double inequality
\begin{equation*}
Q^{\alpha_2}(a,b)C^{1-\alpha_2}(a,b)<AG_{Q,C}(a,b)<Q^{\beta_2}(a,b)C^{1-\beta_2}(a,b)
\end{equation*}
holds for all $a,b>0$ with $a\neq b$ if and only if $\alpha_2\geq1/2$ and $\beta_2\leq\delta_2=0.4784\cdots$.
\end{theorem}

\begin{proof}
Without loss of generality, we assume that $a>b>0$. Let $r=(a-b)/\sqrt{2(a^2+b^2)}\in(0,\sqrt{2}/2)$, then from \eqref{eq:3.1} and \eqref{eq:3.2} we clearly see that
\begin{equation}\label{eq:3.4}
\frac{\log C(a,b)-\log AG_{Q,C}(a,b)}{\log C(a,b)-\log Q(a,b)}=g(r),
\end{equation}
where $g(r)$ is defined as in Lemma \ref{lm2.4}.

Therefore, Theorem \ref{tm3.2} follows directly from \eqref{eq:3.4} and Lemma \ref{lm2.4}.

\end{proof}

\begin{theorem}\label{tm3.3}
The double inequality
\begin{equation*}
\frac{Q(a,b)C(a,b)}{\alpha_3Q(a,b)+(1-\alpha_3)C(a,b)}<AG_{Q,C}(a,b)<\frac{Q(a,b)C(a,b)}{\beta_3Q(a,b)+(1-\beta_3)C(a,b)}
\end{equation*}
holds for all $a,b>0$ with $a\neq b$ if and only if $\alpha_3\leq1/2$ and $\beta_3\geq\delta_3=0.5646\cdots$.
\end{theorem}

\begin{proof}In order to prove the double inequality in Theorem \ref{tm3.3}, it suffices to find $\alpha_3$ and $\beta_3$ such that 
\begin{equation} \label{eq:3.5}
\alpha_3<\frac{\frac{1}{AG_{Q,C}(a,b)}- \frac{1}{Q(a,b)}}{\frac{1}{C(a,b)}-\frac{1}{Q(a,b)}}<\beta_3
\end{equation}
holds for all $a,b>0$ with $a\neq b$.
 
Without loss of generality, we assume that $a>b>0$. Let $r=(a-b)/\sqrt{2(a^2+b^2)}\in(0,\sqrt{2}/2)$, then  \eqref{eq:3.1} and \eqref{eq:3.2} lead to 
\begin{equation}\label{eq:3.6}
\frac{\frac{1}{AG_{Q,C}(a,b)}- \frac{1}{Q(a,b)}}{\frac{1}{C(a,b)}-\frac{1}{Q(a,b)}}=h(r),
\end{equation}
where $h(r)$ is defined as in Lemma \ref{lm2.5}.

Therefore, Theorem \ref{tm3.3} follows directly from \eqref{eq:3.5}, \eqref{eq:3.6} and Lemma \ref{lm2.5}.
\end{proof}

\begin{theorem}\label{tm3.4}
Let $\alpha_4,\beta_4\in(1/2,1)$, then the double inequality
\begin{align*}
C\left(\sqrt{\alpha_4a^2+(1-\alpha_4)b^2},\sqrt{(1-\alpha_4)a^2+\alpha_4b^2}\right)&<AG_{Q,C}(a,b)\\
&\hspace{1cm}<C\left(\sqrt{\beta_4a^2+(1-\beta_4)b^2},\sqrt{(1-\beta_4)a^2+\beta_4b^2}\right)
\end{align*}
holds for all $a,b>0$ with $a\neq b$ if and only if $\alpha_4\leq (\sqrt{2}+2)/4$ and $\beta_4\geq(\delta_4+2)/4=0.9597\cdots$.
\end{theorem}

\begin{proof}
Since $AG(a,b)$ and $C(a,b)$ are symmetric and homogeneous of degree one, we assume that $a>b>0$. Let $r=(a-b)/\sqrt{2(a^2+b^2)}\in(0,\sqrt{2}/2)$, then \eqref{eq:3.2} and the definition of $C(a,b)$ lead to 
\begin{align}\nonumber
&AG_{Q,C}(a,b)-C\left(\sqrt{pa^2+(1-p)b^2},\sqrt{pb^2+(1-p)a^2}\right)\\ \nonumber
&=C(a,b)\left[\frac{\pi}{2\mathcal{K}(r)}-\frac{\sqrt{1+(4p-2)rr'}-\sqrt{1-(4p-2)rr'}}{(4p-2)r}\right]\\ \label{eq:3.7}
&=C(a,b)\left[\frac{\pi}{2\mathcal{K}(r)}-\Phi_{4p-2}(r)\right]
\end{align}
where $\Phi_\lambda(r)$ is defined as in Lemma \ref{lm2.9}. 

It is easy to be verified that $C\left(\sqrt{pa^2+(1-p)b^2},\sqrt{pb^2+(1-p)a^2}\right)$ is continuous and strictly increasing on $[1/2,1]$ with respect to $p$ for fixed $a,b>0$ with $a\neq b$. 

\medskip

We divide the proof into three cases.

\medskip

{\it Case 1.}  $p_1=(\sqrt{2}+2)/4$.  We clearly see from \cite[Lemma 2.8 (1)]{JDZ} that 
\begin{equation}\label{eq:3.8}
\Phi_{\sqrt{2}}(r)<1-\frac{r^2}{4}-\frac{r^4}{4}
\end{equation}for $r\in(0,\sqrt{2}/2)$.

It follows from Corollary \ref{cl2.8} and \eqref{eq:3.8} that 
\begin{align}\nonumber
\frac{\pi}{2\mathcal{K}(r)}-\Phi_{4p_1-2}(r)&=\frac{\pi}{2\mathcal{K}(r)}-\Phi_{\sqrt{2}}(r)\\ \label{eq:3.9}
&>1-\frac{r^4}{4}-\frac{7r^4}{64}-\left(1-\frac{r^4}{4}-\frac{r^4}{4}\right)=\frac{9r^4}{64}>0
\end{align}for $r\in(0,\sqrt{2}/2)$. 

Therefore,  $AG_{Q,C}(a,b)>C\left(\sqrt{p_1a^2+(1-p_1)b^2},\sqrt{(1-p_1)a^2+p_1b^2}\right)$
follows from \eqref{eq:3.7} and \eqref{eq:3.9}. 

\medskip

{\it Case 2.}  $p_2=(\delta_4+2)/4$. Then from Corollary \ref{cl2.8} and Lemma \ref{lm2.9} \ref{it2.8.1} we clearly see that 
\begin{align}\nonumber
\frac{\pi}{2\mathcal{K}(r)}-\Phi_{4p_2-2}(r)&=\frac{\pi}{2\mathcal{K}(r)}-\Phi_{\delta_4}(r)\\ \label{eq:3.10}
&<1-\frac{r^4}{4}-\frac{3r^4}{64}-\left(1-\frac{r^4}{4}\right)=-\frac{3r^4}{64}<0
\end{align}for $r\in(0,33/50)$. 

Furthermore, it follows from Lemma \ref{lm2.9} \ref{it2.8.2} and \ref{it2.8.3} that 
$\mathcal{K}(r)\Phi_{\delta_4}(r)=\left[\mathcal{K}(r)/\sqrt{r}\right]\cdot\left[\sqrt{r}\Phi_{\delta_4}(r)\right]$
is strictly decreasing on $(33/50,\sqrt{2}/2)$.  As a consequence, 
\begin{equation}\label{eq:3.11}
\mathcal{K}(r)\Phi_{\delta_4}(r)>\mathcal{K}(\sqrt{2}/2)\Phi_{\delta_4}(\sqrt{2}/2)=\frac{\pi}{2}
\end{equation}
for $r\in(0,33/50)$.

It follows from \eqref{eq:3.11} that
\begin{equation}\label{eq:3.12}
\frac{\pi}{2\mathcal{K}(r)}-\Phi_{4p_2-2}(r)=\frac{\pi/2-\mathcal{K}(r)\Phi_{\delta_4}(r)}{\mathcal{K}(r)}<0
\end{equation}for $r\in(33/50,\sqrt{2}/2)$. 

Therefore,  $AG_{Q,C}(a,b)<C\left(\sqrt{p_2a^2+(1-p_2)b^2},\sqrt{(1-p_2)a^2+p_2b^2}\right)$
follows from \eqref{eq:3.7}, \eqref{eq:3.10} and \eqref{eq:3.12}. 

\medskip

{\it Case 3.}  $(\sqrt{2}+2)/4<p_3<(\delta_4+2)/4$.  On the one hand, if $r\rightarrow0$, then making use of Taylor series yields
\begin{equation}\label{eq3.13}
\frac{\pi}{2\mathcal{K}(r)}-\Phi_{4p_3-2}(r)=-2\left[\left(p_3-\frac{2-\sqrt{2}}{4}\right)\left(p_3-\frac{\sqrt{2}+2}{4}\right)\right]r^2+o(r^4).
\end{equation}
Equations \eqref{eq:3.7} and \eqref{eq3.13} lead to the conclusion that there exists small enough $\tau_1\in(0,\sqrt{2}/2)$ such that 
$AG_{Q,C}(a,b)<C\left(\sqrt{p_3a^2+(1-p_3)b^2},\sqrt{(1-p_3)a^2+p_3b^2}\right)$ for all $a>b>0$ with $(a-b)/\sqrt{2(a^2+b^2)}\in(0,\tau_1)$.

On the other hand,  it follows from
\begin{equation*}
\frac{d[\sqrt{p}+\sqrt{1-p}]}{dp}=\frac{1-2p}{2\sqrt{p(1-p)}(\sqrt{p}+\sqrt{1-p})}<0
\end{equation*}for $p\in(1/2,1)$
 that 
\begin{equation*}
\frac{\pi}{2\mathcal{K}(\sqrt{2}/2)}-\Phi_{4p_3-2}(\sqrt{2}/2)=\frac{\pi}{2\mathcal{K}(\sqrt{2}/2)}-\frac{1}{\sqrt{p}+\sqrt{1-p}}
\end{equation*} is strictly decreasing on $(1/2,1)$ with respect to $p$. This implies that
\begin{equation}\label{eq:3.14}
\frac{\pi}{2\mathcal{K}(\sqrt{2}/2)}-\Phi_{4p_3-2}(\sqrt{2}/2)<\frac{\pi}{2\mathcal{K}(\sqrt{2}/2)}-\Phi_{\delta_4}(\sqrt{2}/2)=0.
\end{equation}
Equations \eqref{eq:3.7} and \eqref{eq:3.14} lead to the conclusion that there exists small enough $\tau_2\in(0,\sqrt{2}/2)$ such that $AG_{Q,C}(a,b)>C\left(\sqrt{p_3a^2+(1-p_3)b^2},\sqrt{(1-p_3)a^2+p_3b^2}\right)$ for all $a>b>0$ with $(a-b)/\sqrt{2(a^2+b^2)}\in(\sqrt{2}/2-\tau_2,\sqrt{2}/2)$.
\end{proof}

\medskip

\section{Applications}

In this section, we will present new bounds for the complete elliptic integrals $\mathcal{K}(r)$ and $\mathcal{E}(r)$ on $(0,\sqrt{2}/2)$.

Theorem \ref{th4.1} follows from Theorem \ref{tm3.1}, \ref{tm3.2}, \ref{tm3.3} and \ref{tm3.4} immediately. 

\begin{theorem}\label{th4.1}
Let $r'=\sqrt{1-r^2}$ and 
\begin{align*}
m(r)&=\max\left\{\frac{2}{1+r'},r'^{-\delta_2},\frac{1-\delta_3+\delta_3r'}{r'},\frac{\sqrt{1+\delta_4 rr'}+\sqrt{1-\delta_4 rr'}}{2r'}\right\},\\
M(r)&=\min\left\{\frac{1}{1-\delta_1+\delta_1r'},\frac{1}{\sqrt{r'}},\frac{1+r'}{2r'},\frac{\sqrt{1+\sqrt{2} rr'}+\sqrt{1-\sqrt{2} rr'}}{2r'}\right\},
\end{align*}where $\delta_1,\delta_2,\delta_3$ and $\delta_4$ are defined as in Lemmas \ref{lm2.3}, \ref{lm2.4}, \ref{lm2.5} and \ref{lm2.9}, respectively.
Then the double inequality 
\begin{equation*}
\frac{\pi}{2}m(r)<\mathcal{K}(r)<\frac{\pi}{2}M(r)
\end{equation*}holds for all $r\in(0,\sqrt{2}/2)$.
\end{theorem}

Observed that the double inequality 
\begin{equation}\label{eq:4.1}
\frac{\pi^2}{4}<\mathcal{E}(r)\mathcal{K}(r)<\frac{\pi^2}{4\sqrt{r'}}
\end{equation}for $0<r<1$ was presented in  \cite{AVV}. It follows easily from \eqref{eq:4.1} that
\begin{equation}\label{eq:4.2}
\frac{\pi^2}{4\mathcal{K}(r)}<\mathcal{E}(r)<\frac{\pi^2}{4\sqrt{r'}\mathcal{K}(r)}
\end{equation}for $r\in(0,1)$.

The following theorem is derived from Theorem \ref{th4.1} and \eqref{eq:4.2} immediately.

\begin{theorem}
Suppose that $m(r), M(r)$ are defined as in Theorem \ref{th4.1}, then  the double inequality 
\begin{equation*}
\frac{\pi}{2M(r)}<\mathcal{E}(r)<\frac{\pi}{2\sqrt{r'}m(r)}
\end{equation*}holds for all $r\in(0,\sqrt{2}/2)$.
\end{theorem}

\newpage

\section*{Acknowledgement}
 
I am grateful to Professor Xingjiang Lu, Professor Shoufeng Shen, Doctor Zhengchao Ji for helpful conversations.

\newpage

\section*{Academic Integrity Statement}
My signature below constitutes my pledge that all of the writing is my own work, with the exception of those portions which are properly documented. 
 
 I understand and accept the following definition of plagiarism: 
 
1.Plagiarism includes the literal repetition without acknowledgment of the writings of another author.  All significant phrases, clauses, or passages in this paper which have been taken directly from source material have been enclosed in quotation marks and acknowledged in the text itself as well as in the list of Works Cited or Bibliography. 
 
2.Plagiarism includes borrowing another’s ideas and representing them as my own.  To paraphrase the thoughts of another writer without acknowledgment is to plagiarize. Plagiarism also includes inadequate paraphrasing.  Paraphrased passages (those put into my own words) have been properly acknowledged in the text and in the bibliography. 
 
3.Plagiarism includes using another person or organization to prepare this paper and then submitting it as my own work.

\bigskip


\begin{thebibliography}{widest-label}

\bibitem{AQ} H. Alzer and S.-L. Qiu, \textit{Monotonicity theorems and inequalities for the complete elliptic integrals,} Journal of Computational and Applied Mathematics, {\bf 172} (2004), 289-312.

\bibitem{AVV} G. D. Anderson, M. K. Vamanamurthy, M. K. Vuorinen, \textit{Conformal Invariants, Inequalities, and Quasiconformal Maps}, JohnWiley \& Sons, New York, NY, USA, (1997).

\bibitem{BB}J. M. Borwein and P. B. Borwein, \textit{Inequalities for compound mean iterations with logarithmic asymptotes,} J. Math. Anal. Appl. {\bf 177} (1993),  572-582. 

\bibitem{Ca} B. C. Carlson, \textit{Hidden symmetries of special functions,} SIAM Review, {\bf 12} (1970), 332-345.

\bibitem{CW} Y.-M. Chu and M.-K. Wang, \text{Inequalities between arithmetic-geometric, Gini, and Toader means,}  Abstract and Applied Analysis, {\bf 2012} (2012), Article ID 830585.

\bibitem{CV} B. C. Carlson and M. K.Vuorinen, \textit{Inequality of the AGM and the logarithmic mean,} SIAM Review. {\bf 33}(1991), 653-654.

\bibitem{DZ}Q. Ding and T. H. Zhao, \textit{Optimal bounds for arithmetic-geometric and Toader means in terms of generalized logarithmic mean}, J. Inequal. Appl. {\bf 2017} (2017),  1-12. 

\bibitem{EBJ} I. Elishakoff, V. Birman, and J. Singer, \textit{Influence of initial imperfections on nonlinear free vibration of elastic bars,} Acta Mechanica,  {\bf 55} (1985), 191-202.

\bibitem{Ho} T. Horiguchi, \textit{Lattice Green?s function for anisotropic trianguar lattice,} Physica A, {\bf 178} (1991), 351-363.

\bibitem{JDZ}Z. C. Ji, Q. Ding and T. H. Zhao, \textit{Optimal inequalities for a Toader-type mean by quadratic and contraharmonic means,} Journal of nonlinear Science and Applications, {\bf 11} (2018), 150-160.

\bibitem{Ku}R. K\"{u}hnau, \textit{Eine Methode, die Positivit\"{a}t einer Funktion zu pr\"{u}fen,} Zeitschrift f\"{u}r Angewandte Mathematik und Mechanik,  {\bf 74} (1994), 140-143.

\bibitem{Le} D. K. Lee, \textit{Application of theta functions for numerical evaluation of complete elliptic integrals of the first and second kinds,}  Computer Physics Communications,  {\bf 60}(1990), 319-327.

\bibitem{Ma} C. C. Maican, \textit{Integral Evaluations Using the Gamma and Beta Functions and Elliptic Integrals in Engineering, }International Press, Cambridge, Mass, USA, 2005.

\bibitem{MF} K. Mayrhofer and F. D. Fischer, \textit{Derivation of a new analytical solution for a general two-dimensional finite-part integral applicable in fracture mechanics,} International Journal for Numerical Methods in Engineering, {\bf 33} (1992), 1027-1047.

\bibitem{NS} E. Neuman and J. S\'{a}ndor, \textit{On certain means of two arguments and their extensions,} Int. J. Math. Math. Sci. {\bf 16} (2003), 981-993.

\bibitem{QV} S.-L. Qiu and M. K. Vamanamurthy, \textit{Sharp estimates for complete elliptic integrals,}  SIAM Journal on Mathematical Analysis, {\bf 27} (1996), 823-834.

\bibitem{Sa} J. S\'{a}ndor, \textit{On certain inequalities for means,} J. Math. Anal. Appl. {\bf 189} (1995), 602-606.

\bibitem{VV} M. K. Vamanamurthy and M. K. Vuorinen, \textit{Inequalities for means,} J. Math. Anal. Appl. {\bf 183} (1994), 155-166.

\bibitem{WQC}H. Wang, W. M. Qian and Y. M. Chu, \textit{Optimal bounds for Gaussian Arithmetic-Geometric mean with applications to complete elliptic integral,} Journal of Function Spaces, {\bf 2016} (2016), Article ID 3698463.

\bibitem{Ya} Z. H. Yang, \textit{A new. proof of inequalities for Gauss compound mean,} Int. J. Math. Anal. {\bf 4} (2010), 1013-1018.

\bibitem{YCZ} Z.-H. Yang, Y.-M. Chu and W. Zhang, \textit{Accurate approximations for the complete elliptic integral of the second kind,} Journal of Mathematical Analysis and Applications, {\bf 438} (2016), 875-888.
 
\bibitem{YSC} Z.-H. Yang, Y.-Q. Song and Y.-M. Chu, \textit{Sharp bounds for the arithmetic-geometric mean,} Journal of Inequalities and Applications, {\bf 2014} (2014), article 192.


\end{thebibliography}
\end{document}